\documentclass[12pt,amscd]{amsart}
\usepackage[all]{xy}
\usepackage{graphicx}
\usepackage{amsmath,amsxtra,amssymb,latexsym, amscd,amsthm}
\usepackage{indentfirst}
\usepackage[mathscr]{eucal}

\newtheorem{thm}{Theorem}[section]
\newtheorem{cor}[thm]{Corollary}
\newtheorem{lem}[thm]{Lemma}

\theoremstyle{definition}
\newtheorem{defn}[thm]{Definition}

\newtheorem{rem}[thm]{Remark}

\DeclareMathOperator{\N}{\mathbb {N}}
\DeclareMathOperator{\Z}{\mathbb {Z}}
\DeclareMathOperator{\R}{\mathbb {R}}

\DeclareMathOperator{\rank}{rank}
\DeclareMathOperator{\reg}{reg}

\def\alb {\boldsymbol {\alpha}}
\def\btb {\boldsymbol {\beta}}

\def\gmb {\boldsymbol {\gamma}}

\def\x {\mathbf x}

\def\mi {\mathfrak m}
\def\ni {\mathfrak n}
\def\P {\mathcal P}
\def\C {\mathcal C}

\def\h {\widetilde{H}}

\def\H {\mathcal{H}}
\def\V {\mathcal{V}}
\def\E {\mathcal{E}}

\begin{document}

\title[Regularity of powers of cover ideals of unimodular hypergraphs] {Regularity of powers of cover ideals of unimodular hypergraphs}

\author{Nguyen Thu Hang}
\address{Thai Nguyen College of Sciences, Thai Nguyen University, Thai Nguyen, Vietnam}
\email{nguyenthuhang0508@gmail.com}

\author{Tran Nam Trung}
\address{Institute of Mathematics, VAST, 18 Hoang Quoc Viet, Hanoi, Viet Nam}
\email{tntrung@math.ac.vn}

\date{}

\dedicatory{Dedicated to Professor Le Tuan Hoa on his 60th birthday}
\commby{}
\begin{abstract} Let $\H$ be a unimodular hypergraph over the vertex set $[n]$ and let $J(\H)$ be the cover ideal of $\H$ in the polynomial ring $R=K[x_1,\ldots,x_n]$. We show that $\reg J(\H)^s$ is a linear function in $s$ for all $s\geqslant r\left\lceil \frac{n}{2}\right\rceil+1$ where $r$ is the rank of $\H$. Moreover for every $i$, $a_i(R/J(\H)^s)$ is also a linear function in $s$ for $s \geqslant n^2$. 
\end{abstract}

\maketitle
\section*{Introduction}

Let $R := K[x_1,\ldots,x_n]$ be a polynomial ring over a field $K$ and  $\mi := (x_1,\ldots,x_n)$ the maximal homogeneous ideal of $R$. Let $M$ be a finitely generated graded $R$-module. For each $i=0,\ldots,\dim M$ we define the $a_i$-{\it invariant} of $M$ by
$$a_i(M) :=\max\{t\mid H_{\mi}^i(M)_t \ne 0\},$$
where $H_{\mi}^i(M)$ is the $i$-th local cohomology module of $M$ with support in $\mi$ (with the convention $\max \emptyset = -\infty$); and the {\it Castelnuovo-Mumford regularity} (or {\it regularity} for short) of $M$ is defined by
$$\reg(M) := \max\{a_i(M)+i \mid i=0,\ldots,\dim M\}.$$

Let $I$ be a homogeneous ideal in  $R$. It is well-known that $\reg I^s$ is a linear function in $s$ for $s$ big enough (see \cite{CHT,K, TW}). More precisely, there are non-negative integers $d, e$ and $s_0$ such that
$$\reg I^s = ds+e \text{ for all } s\geqslant s_0.$$

While $d$ can be determined in terms of generators of $I$ (see \cite{K}), there are no good interpretations for $e$ and $s_0$. Two natural questions arise from  this result (see \cite{EU}):
\medskip
\begin{enumerate}
\item What is the nature of the number $e$?
\item What is a reasonable bound for $s_0$?
\end{enumerate}
\medskip

These problems are continue to attract us (see \cite{ABS, Ba, Ber, BHT, C, EH, EU, Ha, HTT, JNS, MT2}). When $I$ is generated by forms of the same degree, there is a geometric interpretation for $e$ (see \cite{C, EH, Ha}). Effective bounds of $s_0$ are known only for a few special classes of ideals $I$, such as edge ideals of forests and unicyclic graphs (see \cite{ABS, BHT, JNS}), $\mi$-primary ideals (see \cite{Ber, C1}).

Notice that
$$\reg I^s = 1 + \reg R/I^s =1+\max\{a_i(R/I^s)+i \mid i=0,\ldots,\dim R/I\},$$
it is also natural to ask whether $a_i(R/I^s)$ is asymptotically linear in $s$ or not. However, there is an example such that $\lim_{s\to \infty}\dfrac{\reg \widetilde{I^s}}{s}$ is an irrational number (see \cite{Ckk}), so that $a_i(R/I^s)$ need not to be linear when $s$ large. If we assume furthermore that $I$ is a monomial ideal then $a_i(R/I^s)$ is a quasi-linear function in $s$ due to \cite{HT2}, but we still do not know whether $a_i(R/I^s)$ is asymptotically linear  in $s$ or not.

In this paper when $I$ is the cover ideal of a unimodular hypergraph, we will prove that $a_i(R/I^s)$ is actually asymptotically linear  in $s$. Using this result we are able to give a reasonable bound of $s_0$ for which $\reg I^s$ is a linear function in $s$ for all $s\geqslant s_0$.

Before stating our result we recall some terminology from graph theory (see \cite{Berge} for more detail).  Let $\V = [n]:=\{1,\ldots,n\}$, and let $\E$ be a family of distinct nonempty subsets of $\V$. The pair $\H=(\V,\E)$ is called {\it a hypergraph} with vertex set $\V$ and edge set $\E$. The rank of  $\H$, denoted by $\rank(\H)$, is the maximum cardinality of any of the edges in $\H$. Notice that a hypergraph generalizes the classical notion of a graph; a graph is a hypergraph of rank at most $2$. One may also define a hypergraph by its incidence matrix $A(\H) = (a_{ij})$, with columns representing the edges $E_1,E_2, \ldots,E_m$ and rows representing the vertices $1,2,\ldots, n$ where $a_{ij} = 0$ if $i\notin E_j$ and  $a_{ij} = 1$ if $i\in E_j$. A hypergraph $\H$ is said to be {\it unimodular} if its incident matrix is {\it totally unimodular}, i.e., every square submatrix of $A(\H)$ has determinant equal to $0, 1$ or $-1$. 

A {\it vertex cover} of $\H$ is a subset of $\V$ which meets every edge of $\H$; a vertex cover is {\it minimal} if none of its proper subsets is itself a cover. For a subset  $\tau = \{i_1,\ldots,i_t\}$ of $\V$, set $\x_{\tau} := x_{i_1}\cdots x_{i_t}$. The {\it cover ideal} of $\H$ is then defined by 
$$J(\H) := (\x_{\tau} \mid \tau \text{ is a minimal vertex cover of } \H).$$

It is well-known that there is one-to-one correspondence between squarefree monomial ideals of $R$ and cover ideals of hypergraphs on the vertex set $\V$.

\medskip

The first main result of this paper is the following theorem.

\medskip

\noindent{\bf Theorem \ref{T1}} {\it For any $i$, we have either $a_i(R/J(\H)^s)=-\infty$ for all $s\geqslant 1$ or  there are positive integers $d$ and $e$ with $d\leqslant e$ such that $a_i(R/J(\H)^s)=ds-e$ for all $s\geqslant n^2$.
}

\medskip

We next address the question of bounding $s_0$ for which $\reg I^s$ is a linear function in $s$ whenever $s\geqslant s_0$.   Let $d(J(\H))$ be the maximal degree of minimal monomial generators of $J(\H)$.  Then, the main result of the paper is the following theorem.

\medskip

\noindent{\bf Theorem \ref{T2}} {\it Let $\H$ be a unimodular hypergraph with $n$ vertices and rank $r$. Then there is a non-negative integer $e\leqslant \dim R/J(\H)-d(J(\H))+1$ such that $\reg J(\H)^s = d(J(\H))s+e$ for all $s\geqslant r\left\lceil\frac{n}{2}\right\rceil+1$.
}

\medskip

When $\H$ is a bipartite graph $G$ with $n$ vertices, it is unimodular by \cite[Theorem $5$]{Berge}. The theorem \ref{T2} now says that $\reg(J(G)^s)$ is a linear function in $s$ for all $s\geqslant n+2$.

\medskip

Our approach is based on a generalized Hochster's formula for computing local cohomology modules of arbitrary monomial ideals formulated by Takayama \cite{T}.  Using this formula we are able to investigate the $a_i$-invariants of powers of monomial ideals via the integer solutions of certain systems of linear inequalities. This allows us to use the theory of integer programming as the key role in this paper (see e.g. \cite{HT1, HKTT, HT2} for this approach).

\medskip

The paper is organized as follows. In Section $1$, we set up some basic notation and terminology for simplicial complex, the relationship between simplicial complexes and cover ideals of hypergraphs; and give a generalization of Hochster's formula for computing local cohomology modules. In Section $2$, we investigate the integer solutions of systems of linear inequalities with totally unimodular matrices. In Section $3$,  we prove that $a_i(R/J(\H)^s)$ is an asymptotically linear function in $s$, and settle the problems when $\reg J(\H)^s$  becomes a linear function in $s$, where $\H$ is a unimodular hypergraph.

\section{Preliminary}

Let $K$ be a field and let $R := K[x_1,\cdots,x_n]$ be a polynomial ring over $n$ variables. We first recall a relationship between cover ideals of hypergraphs and simplicial complexes. A {\it simplicial complex} on $\V = \{1,\ldots, n\}$ is a collection of subsets of $\V$ such that if $\sigma\in \Delta $ and $\tau\subseteq \sigma$ then $\tau\in \Delta$. 

\begin{defn} The Stanley-Reisner ideal associated to a simplicial complex $\Delta$ is the squarefree monomial ideal
$$I_{\Delta}:=(x_\tau \mid \tau \notin \Delta) \subseteq R.$$
\end{defn}
Note that  if $I$ is a squarefree monomial ideal, then it is a Stanley-Reisner ideal of the simplicial complex $\Delta(I) := \{ \tau \subseteq \V \mid \x_{\tau} \notin I\}$. If $I$ is a monomial ideal (maybe not squarefree) we also use $\Delta(I)$ to denote the simplicial complex corresponding to the squarefree monomial ideal $\sqrt{I}$.

Let $\mathcal{F}(\Delta)$ be the set of facets of $\Delta$. If $\mathcal{F}(\Delta) =\{F_1,\ldots,F_m\}$, we write $\Delta = \left<F_1,\ldots,F_m\right>$.  Then, $I_{\Delta}$ has the primary-decomposition (see \cite[Theorem $1.7$]{MS}):
$$I_{\Delta} = \bigcap_{F\in\mathcal F(\Delta)} (x_i\mid i\notin F).$$

For $n\geqslant 1$, the $n$-th symbolic power of $I_{\Delta}$ is
$$I_{\Delta}^{(n)} = \bigcap_{F\in\mathcal F(\Delta)} (x_i\mid i\notin F)^n.$$

Let $\H = (\V, \E)$ be a hypergraph. Then, the cover ideal of $\H$ can be written as
\begin{equation} \label{intersect}
J(\H) = \bigcap_{E\in\E} (x_i\mid i\in E).
\end{equation}
By this formula, when considering $J(\H)$ without loss of generality we may assume that $\H$ is {\it simple }, i.e., whenever $E_i, E_j \in \E$ and $E_i \subseteq E_j$, then $E_i = E_j$. In this case,  $J(\H)$ is a Stanley-Reisner ideal with 
\begin{equation}\label{complex-cover}
\Delta(J(\H)) = \left<\V\setminus E \mid E\in\E\right>.
\end{equation}

Let $I$ be a non-zero monomial ideal. Since $R/I$ is an $\mathbb N^n -$ graded algebra, $H^i_{\frak m}(R/I)$ is an $\mathbb Z^n$-graded module over $R/I$ for every $i$. For each degree $\alb=(\alpha_1,\ldots,\alpha_n)\in\Z^n$, in order to compute $\dim_K H_{\mi}^i(R/I)_{\alb}$ we use a formula given by Takayama \cite[Theorem $2.2$]{T} which is a generalization of Hochster's formula for the case $I$ is squarefree \cite[Theorem 4.1]{HO}.

Set $G_{\alb}:=\{i\mid \alpha_i<0\}$. For a subset $F\subseteq \V$, we let $R_F:=R[x_i^{-1}\mid i\in F\cup G_{\alb}]$. Define the simplicial complex $\Delta_{\alb}(I)$ by
\begin{equation}\label{degree-complex}
\Delta _{\alpha }(I) :=\{F\subseteq \V\setminus G_{\alb}\mid x^{\alpha }\notin IR_F\}.
\end{equation}

\begin{lem} \label{TA}\cite[Theorem 2.2]{T}  $\dim_K {H_{\frak m}^i(R/I)_{\alb}}=\dim_K \widetilde{H}_{i-\mid G_{\alpha }\mid-1 }(\Delta _{\alb}(I);K).$ 
\end{lem}

If $\H$ is unimodular, the cover ideal $J(\H)$ is normally torsion-free, i.e. $J(\H)^{(s)} = J(\H)^s$ for all $s\geqslant 1$ by \cite[Theorem 1.1]{HHT}. Combining with \cite[Lemma $1.3$]{MT1} we obtain:

\begin{lem} \label{uni-complex}Let $\H = (\V,\E)$ be a unimodular hypergraph with $\V =\{1,\ldots,n\}$, and $\alb=(\alpha_1,\ldots,\alpha_n)\in\N^n$. Then, for every $s\geqslant 1$ we have
$$
\Delta _{\alb}(J(\H)^s)=\left<\V\setminus E\mid E\in \E \text{ and } \sum_{i\in E} \alpha_i \leqslant s-1\right>.
$$
\end{lem}

\section{Integer polytopes}

For a vector $\alb = (\alpha_1,\ldots,\alpha_n)\in\R^n$, we set $|\alb| : = \alpha_1+\cdots+\alpha_n$ and for a nonempty bounded closed subset $S$ of $\R^n$ we set 
$$\delta(S) : = \max\{|\alb| \mid \alb \in S\}.$$

Let $\H=(\V,\E)$ be a unimodular hypergraph on the vertex set $\V =\{1,\ldots,n\}$. Assume that 
$$H_{\mi}^p(R/J(\H)^q)_{\btb} \ne 0$$
for some $p\geqslant 0$, $q\geqslant 1$ and $\btb = (\beta_1,\ldots,\beta_n)\in\N^n$.

By Lemma $\ref{TA}$ we have
\begin{equation}\label{N1}
\dim_K \h_{p-1}(\Delta_{\btb}(J(\H)^q);K) =\dim_K H_{\mi}^p(R/J(\H)^q)_{\btb} \ne 0.
\end{equation}
In particular, $\Delta_{\btb}(J(\H)^q)$ is not acyclic.

Suppose that $\E =\{E_1,\ldots,E_m\}$ where $m\geqslant 1$. Then, by Equation $(\ref{complex-cover})$ $$\Delta(J(\H)) = \left<\V\setminus E_1,\ldots, \V\setminus E_m\right>.$$
Since $\Delta_{\btb}(J(\H)^q)$ is not acyclic, by Lemma $\ref{uni-complex}$ we may assume that 
$$\Delta_{\btb}(J(\H)^q) =\left<\V\setminus E_1,\ldots, \V\setminus E_k\right>$$
where $1\leqslant k\leqslant m$. 

For each integer $t\geqslant 1$, let $\mathcal P_t$ be the set of solutions in $\R^n$ of the following system:
\begin{equation}\label{EQ-basics}
\begin{cases}
\sum_{i\in E_j} x_i  \leqslant t-1 & \text{ for } j = 1,\ldots,k,\\
\sum_{i\in E_j} x_i  \geqslant  t & \text{ for } j = k+1,\ldots,m,\\
x_1\geqslant 0,\ldots,x_n\geqslant 0.
\end{cases}
\end{equation}
Then, $\btb \in \mathcal P_q$. Moreover, by Lemma $\ref{uni-complex}$ one has
$$\Delta_{\alb}(J(\H)^t) = \left<\V\setminus E_1,\ldots, \V\setminus E_k\right> = \Delta_{\btb}(J(\H)^q) \ \text{ whenever } \alb\in \P_t \cap \N^n .$$

In order to investigate the set $\P_t$ we consider $\C_t$ to be the set of solutions in $\R^n$ of the following system:
\begin{equation}\label{EQ-polytope}
\begin{cases}
\sum_{i\in E_j} x_i  \leqslant t & \text{ for } j = 1,\ldots,k,\\
\sum_{i\in E_j} x_i  \geqslant  t & \text{ for } j = k+1,\ldots,m,\\
x_1\geqslant 0,\ldots,x_n\geqslant 0.
\end{cases}
\end{equation}

Since $\H$ is unimodular, we have both $\P_t$ and $\C_t$ are integer convex polyhedra by \cite[Theorem $19.1$]{S}, i.e. all their vertices have integer coordinates.

\begin{lem} \label{R1} \cite[Lemma $2.1$]{HT1} $\C_1$ is a polytope with $\dim \C_1 = n$.
\end{lem}

\begin{rem}\label{PBP} Since $C_t = tC_1$ is also a polytope, it is bounded. Thus, for every $i =1,\ldots,n$, from the system $(\ref{EQ-polytope})$ we imply that $i\in E_j$ for some $1\leqslant j\leqslant k$.
\end{rem}

Observe that $\P_t \subseteq \C_t$, so $\P_t$ is a polytope as well. 

\begin{lem}\label{R2} $\P_n\ne\emptyset$.
\end{lem}
\begin{proof} From the system $(\ref{EQ-basics})$, $\P_n$ is the set of solutions of the following system:
\begin{equation}\label{EQ-Pm}
\begin{cases}
\sum_{i\in E_j} x_i  \leqslant n-1 & \text{ for } j = 1,\ldots,k,\\
\sum_{i\in E_j} x_i  \geqslant  n & \text{ for } j = k+1,\ldots,m,\\
x_1\geqslant 0,\ldots,x_n\geqslant 0.
\end{cases}
\end{equation}

If $k=m$, then the zero vector of $\R^n$ is in $\P_n$, and then $\P_n\ne\emptyset$. 

Assume that $k < m$. From the system $(\ref{EQ-polytope})$ we conclude that $\sum_{i\in E_m} x_i  = 1$ is a supporting hyperplane of $\C_1$. Let $F$ be the facet of $\C_1$ determined by this hyperplane. Now take $n$ vertices of $\C_1$ lying in $F$, say $\alb^1, \ldots,\alb^n$, such that they are affinely independent. Let $\alb := (\alb^1+\cdots+\alb^n)/n \in \C_1$. Then, $\alb$ is a relative interior point of $F$, so that it does not belong to any another facet of $\C_1$. Thus, $\alb$ is a solution of the following system:
\begin{equation}
\begin{cases}
\sum_{i\in E_j} x_i  < 1 & \text{ for } j = 1,\ldots,k,\\
\sum_{i\in E_j} x_i  \geqslant  1 & \text{ for } j = k+1,\ldots,m,\\
x_1\geqslant 0,\ldots,x_n\geqslant 0.
\end{cases}
\end{equation}
Therefore, $n\alb$ is a solution of the following system
\begin{equation}
\begin{cases}
\sum_{i\in E_j} x_i  < n & \text{ for } j = 1,\ldots,k,\\
\sum_{i\in E_j} x_i  \geqslant  n & \text{ for } j = k+1,\ldots,m,\\
x_1\geqslant 0,\ldots,x_n\geqslant 0.
\end{cases}
\end{equation}
Together with the fact that $n\alb\in \N^n$, it yields $n\alb \in \P_n$, and thus $\P_n\ne\emptyset$.
\end{proof}

Since $\C_1$ is a polytope, there is a vertex $\gmb=(\gamma_1,\ldots,\gamma_n)$ of $\C_1$ such that 
$$\delta(\C_1) = |\gmb| = \gamma_1+\cdots+\gamma_n.$$
Let $d:= |\gmb|$. Note that $t\gmb$ is also a vertex of $\C_t$ and $\delta(\C_t) = dt$. Since $\P_t \subseteq \C_t$, we have $\delta(P_t) \leqslant dt$, so we can write
$$\delta(P_t) = dt-e_t \text{ for some integer } e_t \geqslant 0.$$

\begin{rem}\label{POSITIVE} Since $C_1$ is a polytope of dimension $n$, we have $d\geqslant 1$.
\end{rem}

\begin{lem}\label{R3} If $\P_t\ne \emptyset$, then $\P_{t+1}\ne \emptyset$ and $e_t\geqslant e_{t+1}$.
\end{lem}
\begin{proof} Let $\alb = (\alpha_1,\ldots,\alpha_n)\in \P_t$ such that $\delta(\P_t) = |\alb|$. Since $\alb$ is a solution of the system $(\ref{EQ-basics})$, and $\gmb$ is a solution of the system $(\ref{EQ-polytope})$ with $1$ in place of $t$, we deduce that
$$
\begin{cases}
\sum_{i\in E_j} (\alpha_i+\gamma_i)  \leqslant t & \text{ for } j = 1,\ldots,k,\\
\sum_{i\in E_j} (\alpha_i+\gamma_i)  \geqslant  t+1 & \text{ for } j = k+1,\ldots,m.
\end{cases}
$$
In other words, $\alb+\gmb\in \P_{t+1}$. Therefore, $\P_{t+1}\ne\emptyset$ and $\delta(P_{t+1}) \geqslant |\alpha|+|\gamma|$. Since $\delta(\P_{t+1}) = d(t+1)-e_{t+1}$ and $|\alb|+|\gmb| = d(t+1)-e_t$, we have $e_t\geqslant e_{t+1}$.
\end{proof}

For a real number $x$ we denote by $\lceil x \rceil$ the least integer $\geqslant x$. The following lemma plays a key role in the paper. It says that $\delta(\P_t)$ is a linear function in $t$ for $t$ big enough.

\begin{lem}\label{R4} There are non-negative integers $d$ and $e$ with $e \leqslant n^2$ such that
$$\delta(\P_t) = dt-e \text{  for } t\geqslant r\left\lceil \frac{n}{2}\right\rceil+1$$
where $r = \rank(\H)$.
\end{lem}
\begin{proof}

By Lemmas $\ref{R2}$ and $\ref{R3}$ we have $e_n\geqslant e_{n+1}\geqslant\cdots \geqslant 0$. It follows that there is $t_0 \geqslant n$ such that $e_t = e_{t_0}$ for $t\geqslant t_0$. Let $e := e_{t_0}$. Then,
$$\delta(\P_t) = dt-e, \text{ for all } t\geqslant t_0.$$
By Lemma $\ref{R3}$ we have
\begin{equation}\label{UPBa}
\delta(\P_t) \leqslant dt-e \text{ whenever } \P_t\ne\emptyset.
\end{equation}

Let $s$ be an integer such that $s\geqslant \max\{n^2+e+1,t_0\}$. Then, $\delta(\P_s) = ds-e$. Because $\P_s$ is a polytope, we have $\delta(\P_s) = |\alb|$ for some  vertex $\alb$ of $\P_s$. By  the system $(\ref{EQ-basics})$, the polytope $\P_s$ is defined by
\begin{equation}\label{LE}
\begin{cases}
\sum_{i\in E_j} x_i  \leqslant s-1 & \text{ for } j = 1,\ldots,k,\\
\sum_{i\in E_j} x_i  \geqslant  s & \text{ for } j = k+1,\ldots,m,\\
x_1\geqslant 0,\ldots,x_n\geqslant 0.
\end{cases}
\end{equation}
By \cite[Formula $23$ in Page $104$]{S} we can represent $\alb$ as the unique solution of  a system of linear equations of the form
\begin{equation}\label{EQ-alpha}
\begin{cases}
\sum_{i\in E_j} x_i  = s-1 & \text{ for } j \in S_1,\\
\sum_{i\in E_j} x_i  =  s & \text{ for } j \in S_2,\\
x_j=0,  & \text{ for } j \in S_3,
\end{cases}
\end{equation}
where $S_1 \subseteq [k]$, $S_2\subseteq \{k+1,\ldots,m\}$ and $S_3\subseteq [n]$ such that $|S_1|+|S_2|+|S_3| = n$. 

Let $A$ be the matrix of the system $(\ref{EQ-alpha})$ and write $\alb = (\alpha_1,\ldots,\alpha_n)$.  For each $i$, let $A_i$ be the matrix obtained from $A$ by replacing the $i$-th column by: the first $|S_1|$ entries are $s-1$, the next $|S_2|$ consecutive entries are $s$, and the last $|S_3|$ entries are zeroes. Then, by Cramer's rule we have
\begin{equation*}\label{EQ4}\alpha_i = \frac{\det(A_i)}{\det(A)}\  \text{ for } i=1,\ldots, n.
\end{equation*}

Since $\H$ is unimodular, $A$ is totally unimodular, and so $\det(A) = \pm 1$. Without loss of generality, we may assume that $\det(A) = 1$. Thus,
$\alpha_i = \det(A_i), \text{ for } i=1,\ldots,n$.

Write the matrix $A_i$ in the form
 \[A_i= \begin{bmatrix} *&*&\cdots & s-1 & \cdots & *\\
\vdots &\vdots &\cdots & \vdots & \cdots & \vdots\\
*&*&\cdots & s-1 & \cdots & *\\
*&*&\cdots & s & \cdots & *\\
\vdots &\vdots &\cdots & \vdots & \cdots & \vdots\\
*&*&\cdots & s & \cdots & *\\
*&*&\cdots & 0 & \cdots & *\\
\vdots &\vdots &\cdots & \vdots & \cdots & \vdots\\
*&*&\cdots & 0 & \cdots & *
\end{bmatrix},\]
and let
\begin{equation} \label{EQ-DC}
D_i= \begin{bmatrix} *&*&\cdots & 1 & \cdots & *\\
\vdots &\vdots &\cdots & \vdots & \cdots & \vdots\\
*&*&\cdots & 1 & \cdots & *\\
*&*&\cdots & 1 & \cdots & *\\
\vdots &\vdots &\cdots & \vdots & \cdots & \vdots\\
*&*&\cdots & 1 & \cdots & *\\
*&*&\cdots & 0 & \cdots & *\\
\vdots &\vdots &\cdots & \vdots & \cdots & \vdots\\
*&*&\cdots & 0 & \cdots & *
\end{bmatrix} \text{ and }
C_i=\begin{bmatrix} *&*&\cdots & 1 & \cdots & *\\
\vdots &\vdots &\cdots & \vdots & \cdots & \vdots\\
*&*&\cdots & 1 & \cdots & *\\
*&*&\cdots & 0 & \cdots & *\\
\vdots &\vdots &\cdots & \vdots & \cdots & \vdots\\
*&*&\cdots & 0 & \cdots & *\\
*&*&\cdots & 0 & \cdots & *\\
\vdots &\vdots &\cdots & \vdots & \cdots & \vdots\\
*&*&\cdots & 0 & \cdots & *
\end{bmatrix}.
\end{equation}
Then, $\det(A_i) = \det(D_i)s-\det(C_i)$. Let $d_i := \det(D_i)$ and $c_i := \det(C_i)$ so that
$$\alpha_i = d_is-c_i, \text{ for } i = 1,\ldots,n.$$

When deleting the $i$-th column from $C_i$ we get a totally unimodular matrix. Therefore, by expanding the determinant of $C_i$ along the $i$-th column we obtain
\begin{equation}\label{IQ-c}
|c_i| \leqslant |S_1| \text{ , for } i =1,\ldots, n.
\end{equation}

Let $d' := d_1+\cdots+d_n$ and $c' := c_1+\cdots+c_n$. Then, we have $|\alb| = d's-c'$ and $|c'| \leqslant n|S_1| \leqslant n^2$. Since $\delta(\P_s) = ds-e$, we have $ds-e = d's-c'$. Thus, $(d-d')s = e-c'$. Note that $|e-c'| \leqslant e+n^2 < s$. It follows that $d-d' = 0$ and $e-c'=0$, i.e. $d=d'$ and $e=c'$. In particular, $e\leqslant n^2$.

As $\alb=(d_1s-c_1,\ldots,d_ns-c_n) \in \N^n$ and $|c_i| \leqslant |S_1| \leqslant n<s$,  for every $i$ we have
\begin{equation}\label{CL1}
d_i \geqslant 0 \text{, and } d_i \geqslant 1 \text{ if } c_i > 0.
\end{equation}

Since $\alb\in\P_s$, from the system $(\ref{LE})$ we have
$$\left(\sum_{i\in E_j} d_i\right)s-\sum_{i\in E_j} c_i  \leqslant s-1 \text{ for } j=1,\ldots,k.$$
Together with  $\sum_{i\in E_j}|c_i|\leqslant|E_j|n \leqslant rn$ for every $j$ and $s\geqslant rn+1$, it follows that
\begin{equation}\label{CL2}
\sum_{i\in E_j} d_i=
\begin{cases}
0 & \text{ if } \sum_{i\in E_j}c_i  \leqslant 0, \\
1 & \text{otherwise.}
\end{cases}
\end{equation} 
for  $j=1,\ldots,k$. 

Similarly, we also have 
\begin{equation}\label{CL3}
\sum_{i\in E_j} d_i\geqslant 1 \text{, and } \sum_{i\in E_j} d_i\geqslant 2 \text{ if } \sum_{i\in E_j} c_i > 0
\end{equation} 
for  $j=k+1,\ldots,m$.

We first claim that
\begin{equation}\label{D1}
d_i \leqslant 1 \text{ for every } i=1,\ldots, n.
\end{equation}
Indeed, by Remark $\ref{PBP}$, we have $i\in E_j$ for some $j\in [k]$ , and hence  by $(\ref{CL2})$ we have
\begin{equation}\label{D1}
d_i \leqslant \sum_{u\in E_j}d_u \leqslant 1.
\end{equation}

We next claim that 
\begin{equation} \label{CL4}
|c_i|\leqslant \left\lceil \frac{n}{2}\right\rceil \text{ for every } i=1,\ldots,n.
\end{equation}
Indeed, if $|S_1|\leqslant \left\lceil \frac{n}{2}\right\rceil$, then the claim follows from $(\ref{IQ-c})$. We now verify the claim for the case $|S_1|> \left\lceil \frac{n}{2}\right\rceil$, so that $|S_2| <  \left\lceil \frac{n}{2}\right\rceil$. Let $E_i$ be the maxtrix obtained by replacing the $i$-th column of $A_i$ by the first $|S_1|$ entries are zeroes, the next $|S_2|$ consecutive entries are $1$, and the last $|S_3|$ entries are zeroes. Then, by expanding $|E_i|$ along the $i$-th column we get $\mathopen|\det(E_i)| \leqslant |S_2|$. On the other hand, by $(\ref{EQ-DC})$ we have $\det(D_i) = \det(C_i)+\det(E_i)$. Thus, $|c_i|\leqslant |d_i|+\mathopen|\det(E_i)|\leqslant |d_i|+|S_2|$. Together with Claims $(\ref{CL1})$ and $(\ref{D1})$ it yields $|c_i| \leqslant 1+|S_2|\leqslant \left\lceil \frac{n}{2}\right\rceil$,
as claimed.

We now turn to prove the lemma by showing that $\delta(\P_t) = dt-e$ for all $t\geqslant r\left\lceil \frac{n}{2}\right\rceil+1$. Indeed, let $t$ be  such an integer, and define $$\alb(t) := (d_1t-c_1,\ldots,d_nt-c_n) \in \Z^n.$$ 

We prove that $\alb(t)$ satisfies the system $(\ref{EQ-basics})$. Firstly, since $\left\lceil \frac{n}{2}\right\rceil < t$, by $(\ref{CL1})$ and $(\ref{CL4})$ we have $\alb(t)_i\geqslant 0$ for $i=1,\ldots,n$.

Next, for $j=1,\ldots, k$, by $(\ref{CL2})$ we have two possible cases:

\noindent {\it Case $1$}: $\sum_{i\in E_j} d_i = 0$. Together with $(\ref{CL4})$ we obtain
$$\sum_{i\in E_j}\alb(t)_i = \left(\sum_{i\in E_j} d_i\right)t-\sum_{i\in E_j}c_i=-\sum_{i\in E_j}c_i\leqslant r\left\lceil \frac{n}{2}\right\rceil \leqslant t-1.$$

\noindent {\it Case $2$}: $\sum_{i\in E_j} d_i = 1$ and $\sum_{i\in E_j}c_i \geqslant  1$. Then,
$$\sum_{i\in E_j}\alb(t)_i = \left(\sum_{i\in E_j} d_i\right)t-\sum_{i\in E_j}c_i=t-\sum_{i\in E_j}c_i\leqslant t-1.$$

In both cases we have $\sum_{i\in E_j} \alb(t)_i \leqslant t-1$.

Further, for $j = k+1,\ldots,m$, by $(\ref{CL3})$ we also have two possible cases:

\noindent {\it Case $1$}: $\sum_{i\in E_j} d_i = 1$ and $\sum_{i\in E_j}c_i  \leqslant 0$. In this case,
$$\sum_{i\in E_j}\alb(t)_i = \left(\sum_{i\in E_j} d_i\right)t-\sum_{i\in E_j}c_i=t-\sum_{i\in E_j}c_i\geqslant t.$$

\noindent {\it Case $2$}: $\sum_{i\in E_j} d_i \geqslant 2$. Let $\rho := |\{i\in E_j \mid d_i = 1\}|$. By $(\ref{CL1})$ and Claim $(\ref{D1})$ we have $\rho = \sum_{i\in E_j} d_i\geqslant 2$. Note that if $d_i=0$ then $c_i\leqslant 0$ since $\alb\in\N^n$. Together with Claim $(\ref{CL4})$, it gives
\begin{align*}
\sum_{i\in E_j}\alb(t)_i &= \left(\sum_{i\in E_j} d_i\right)t-\sum_{i\in E_j}c_i = \rho t-\sum_{i\in E_j:d_i=1}c_i-\sum_{i\in E_j:d_i=0}c_i\\
&\geqslant \rho t-\sum_{i\in E_j:d_i=1}c_i \geqslant  \rho t-\rho \left\lceil \frac{n}{2}\right\rceil.
\end{align*}
On the other hand $r = \rank \H \geqslant |E_j| \geqslant \rho$, so that
$$(\rho-1)t \geqslant (\rho-1)\left(r\left\lceil \frac{n}{2}\right \rceil+1\right)\geqslant (\rho-1)\left(\rho \left\lceil \frac{n}{2}\right \rceil+1\right)\geqslant 
\rho \left\lceil \frac{n}{2}\right \rceil.$$
It follows that
$$\sum_{i\in E_j}\alb(t)_i \geqslant \rho t-\rho \left\lceil \frac{n}{2}\right\rceil=t+(\rho-1)t-\rho \left\lceil \frac{n}{2}\right\rceil\geqslant t.$$

Hence, $\sum_{i\in E_j}\alb(t)_i \geqslant t$ in both cases. Thus, $\alb(t)$ is a solution of the system $(\ref{EQ-basics})$, and thus $\alb(t)\in \P_t$. In particular,
$$\delta(P_t) \geqslant |\alb(t)| = (d_1+\cdots+d_n)t-(c_1+\cdots+c_n) =d't-c'= dt-e.$$
Together with $(\ref{UPBa})$ we conclude that $\delta(\P_t) = dt-e$, and the lemma follows.
\end{proof}

\begin{lem} \label{R5} Assume that $\delta(\P_t) = dt-e$ for all $t\gg 0$. Then, $d\leqslant e$.
\end{lem}
\begin{proof}
Observe that $\delta(\P_t)$ is the optimal value of the linear programming problem
$$\max(x_1+\cdots+x_n)$$

subject to
$$
\begin{cases}
\sum_{i\in E_j} x_i  \leqslant t-1 & \text{ for } j = 1,\ldots,k,\\
\sum_{i\in E_j} -x_i  \leqslant  -t & \text{ for } j = k+1,\ldots,m,\\
x_1\geqslant 0,\ldots,x_n\geqslant 0.
\end{cases}
$$

Let $A$ be the matrix of the first $k$ inequalities of this system and $B$ the matrix of the next $q:=m-k$ inequalities.  Then, the dual problem is
$$\min((y_1+\cdots+y_k-z_1-\cdots-z_q)t-(y_1+\cdots+y_k))$$

subject to

\begin{equation}\label{dual}
\begin{cases}
(y_1,\ldots,y_k)A+(z_1,\ldots,z_q)B \geqslant \mathbf{1}_n^T,\\
y_1\geqslant 0,\ldots,y_k\geqslant 0, z_1\geqslant 0,\ldots,z_q\geqslant 0,
\end{cases}
\end{equation}

where $\mathbf{1}_n =(1,\ldots,1)\in \R^n$.

Since $(\ref{dual})$ defines a pointed convex polyhedron in $\R^m$, say $\mathcal Q$, we deduce that there is a vertex $(u_1,\ldots,u_k,v_1,\ldots,v_q)$ of $\mathcal Q$ such that in this convex polyhedron
$$\min((y_1+\cdots+y_k-z_1-\cdots-z_q)t-(y_1+\cdots+y_k)) = at-b \text{ for } t\gg 0,$$
where $a =u_1+\cdots+u_k-(v_1+\cdots+v_q)$ and $b=u_1+\cdots+u_k$. Since $u_i\geqslant 0$ and $v_j\geqslant 0$ for every $i$ and $j$, we have $a\leqslant b$.

By the Duality theorem of linear programing (see e.g. \cite[Corollary 7.1g]{S}) we have
$$\max\{x_1+\cdots+x_n\mid (x_1,\ldots,x_n)\in \P_t\} = at-b \text{ for } t\gg 0.$$
Therefore, $\delta(P_t) = dt-e = at-b$ for $t\gg 0$. It follows that $d = a$ and $e = b$. Consequently, $d\leqslant e$, as required.
\end{proof}

\section{Regularity of powers of ideals}

Let $\H=(\V,\E)$ be a unimodular hypergraph on the vertex set $\V =\{1,\ldots,n\}$. Without loss of generality we may assume that $\E\ne\emptyset$ and thus $J(\H)\ne 0$. Let $r:=\rank(\H)$. In this section we will prove that $a_i(R/J(\H)^s)$ is asymptotically linear in $s$; and then settle the problem on bounding $s_0$ such that  $\reg J(\H)^s$ is a linear function for $s\geqslant s_0$.

\begin{lem} \label{main-lem} Assume that $a_p(R/J(\H)^s)\ne -\infty$ for some $p\geqslant 0$ and $s\geqslant 1$. Then,  there are positive integers $d$ and $e$ such that
\begin{enumerate}
\item $d\leqslant e\leqslant n^2$;
\item $a_p(R/J(\H)^t) \geqslant dt-e$ for all $t \geqslant r\left\lceil \frac{n}{2}\right\rceil+1$; and
\item If $s\geqslant r\left\lceil \frac{n}{2}\right\rceil+1$, then $a_p(R/J(\H)^s) = ds-e$.
\end{enumerate}
\end{lem}

\begin{proof}  Since $a_p(R/J(\H)^s) \ne -\infty$, there is $\btb' =(\beta_1,\ldots,\beta_n)\in \Z^n$ such that 
$$H_{\mi}^p(R/J(\H)^s)_{\btb'}\ne 0 \text{ and }  a_p(R/J(\H)^s) = |\btb'|.$$ 
By Lemma $\ref{TA}$ we have
\begin{equation}\label{N1}
\dim_K \h_{p-|G_{\btb'}|-1}(\Delta_{\btb'}(J(\H)^s);K) =\dim_K H_{\mi}^p(R/J(\H)^s)_{\btb'} \ne 0.
\end{equation}
In particular, $\Delta_{\btb'}(J(\H)^s)$ is not acyclic.

If $G_{\btb'} =[n]$, then $\Delta_{\btb'}(J(\H)^s)$ is either $\{\emptyset\}$ or a void complex. Since it is not acyclic, $\Delta_{\btb'}(J(\H)^s) =\{\emptyset\}$. But then by $(\ref{degree-complex})$ we would have $J(\H) = 0$, a contradiction. Therefore, $G_{\btb'} \ne [n]$.

We may assume that $G_{\btb'} =\{m+1,\ldots,n\}$ for some $1\leqslant m\leqslant n$. Set $S := K[x_1,\ldots,x_m]$. Let $\H'$ be the hypergraph on the vertex set $\V' =\{1,\ldots,m\}$ with the edge set $\E' =\{E\in \E \mid E \subseteq \V'\}$. Since $A(\H')$ is a submatrix of $A(\H)$, we have $\H'$ is also a unimodular hypergraph. Moreover, by $(\ref{intersect})$ we obtain:
\begin{equation}\label{EQ-LOCALIZATION}
J(\H)R_{G_{\btb'}} \cap S = J(\H').
\end{equation}

Let $\btb := (\beta_1,\ldots,\beta_m) \in \N^m$. By using Formulas $(\ref{degree-complex})$ and $(\ref{EQ-LOCALIZATION})$ we get
\begin{equation}\label{POWER}
\Delta_{\btb}(J(\H')^t) = \Delta_{\btb'}(J(\H)^t) \text{ for any } t\geqslant 1.
\end{equation}
Together with $(\ref{N1})$, it gives $\h_{p-|G_{\btb'}|-1}(\Delta_{\btb}(J(\H')^s);K)\ne 0$. By Lemma $\ref{TA}$ we get
$$H_{\ni}^{p-|G_{\btb'}|}(S/J(\H')^s)_{\btb}\ne 0,$$
where $\ni = (x_1,\ldots,x_m)$ is the homogeneous maximal ideal of $S$.

Suppose that $\E' =\{E_1,\ldots,E_k\}$ where $k\geqslant 1$. Then, by Equation $(\ref{complex-cover})$ $$\Delta(J(\H')) = \left<\V'\setminus E_1,\ldots, \V'\setminus E_k\right>.$$
By Lemma $\ref{uni-complex}$ we may assume that 
$$\Delta_{\btb}(J(\H')^s) =\left<\V'\setminus E_1,\ldots, \V'\setminus E_q\right>$$
where $1\leqslant q\leqslant k$. 

For each integer $t\geqslant 1$, let $\mathcal P_t$ be the set of solutions in $\R^m$ of the following system:
\begin{equation}\label{Qt}
\begin{cases}
\sum_{i\in E_j} x_i  \leqslant t-1 & \text{ for } j = 1,\ldots,q,\\
\sum_{i\in E_j} x_i  \geqslant  t & \text{ for } j = q+1,\ldots,k,\\
x_1\geqslant 0,\ldots,x_m\geqslant 0.
\end{cases}
\end{equation}
Then, $\btb \in \mathcal P_s$. From $(\ref{Qt})$ and Lemma $\ref{uni-complex}$ one has
\begin{equation}\label{EQ001}
\Delta_{\alb}(J(\H')^t) = \left<\V\setminus E_1,\ldots, \V\setminus E_q\right> = \Delta_{\btb}(J(\H')^s) \ \text{ whenever } \alb\in \P_t \cap \N^m .
\end{equation}

Together Lemmas $\ref{R4}$ and $\ref{R5}$ with Remark $\ref{POSITIVE}$, there are positive integers $d$ and $f$ with $d\leqslant f\leqslant m^2$ such that
$$\delta(\P_t) = dt-f \text{ , for all } t  \geqslant r\left\lceil \frac{m}{2}\right\rceil+1.$$

For any $t\geqslant r\left\lceil \frac{n}{2}\right\rceil+1$, we also have $\delta(\P_t) =dt-f$. Let $\alb=(\alpha_1,\ldots,\alpha_m)$ be a vertex of $\P_t$ such that $\delta(\P_t) = |\alb|$. Since $A(\H')$ is totally unimodular, $\alb \in \N^m$.

Let $\alb' = (\alpha_1,\ldots,\alpha_m,-1,\ldots,-1)\in \Z^n$. Then, $G_{\alb'} = G_{\btb'}$, and by $(\ref{POWER})$ and $(\ref{EQ001})$ we obtain
$$\Delta_{\alb'}(J(\H)^t) = \Delta_{\alb}(J(\H')^t) = \Delta_{\btb}(J(\H')^s) =  \Delta_{\btb'}(J(\H)^s).$$
By Lemma $\ref{TA}$ we have
\begin{align*}
\dim_K H_{\mi}^p(R/J(\H)^t)_{\alb'} &= \dim_K \h_{i-|G_{\alb'}|-1}(\Delta_{\alb'}(J(\H)^t);K)\\
&=\dim_K \h_{i-|G_{\btb'}|-1}(\Delta_{\btb'}(J(\H)^s);K)\ne 0.
\end{align*}
Consequently, $H_{\mi}^p(R/J(\H)^t)_{\alb'}\ne 0$, and so
$$a_p(R/J(\H)^t) \geqslant |\alb'| = |\alb|-(n-m) = dt -(f+n-m) = dt-e,$$
where we set $e := f+n-m$. Note that $d\leqslant f \leqslant e \leqslant m^2+(n-m)\leqslant n^2$. 

This argument shows that $\btb' = (\beta_1,\ldots,\beta_m,-1,\ldots,-1)$ and $|\btb| = \delta(\P_s)$. Therefore, if $s \geqslant r\left\lceil \frac{n}{2}\right\rceil+1$ then $|\btb| = \delta(\P_s) = ds-f$ and 
$$a_p(R/J(\H)^s) = |\btb'| = |\btb|-(n-m) = ds-(f+n-m) = ds-e,$$
and the lemma follows.
\end{proof}

We now prove that $a_i(R/J(\H)^s)$ is asymptotically linear in $s$.

\begin{thm}\label{T1} For any $i$, we have either $a_i(R/J(\H)^s)=-\infty$ for all $s\geqslant 1$ or  there are positive integers $d$ and $e$ with $d\leqslant e$ such that $a_i(R/J(\H)^s)=ds-e$ for all $s\geqslant n^2$.
\end{thm}
\begin{proof} If $n=1$, then $R=K[x_1]$ and $J(\H) = (x_1)$, and then the theorem holds for this case. We therefore may assume that $n\geqslant 2$.

 Assume that $a_i(R/J(\H)^k)\ne -\infty$ for some $k\geqslant 1$. By Lemma \ref{main-lem} we have $a_i(R/J(\H)^t)\ne-\infty$ for every $t\geqslant r\left\lceil \frac{n}{2}\right\rceil+1$.

Let $s_0$ be an integer such that $s_0\geqslant n^2$. Note that $r\left\lceil \frac{n}{2}\right\rceil+1\leqslant n^2$ as $n\geqslant 2$, so $a_i(R/J(\H)^{s_0})\ne -\infty$. By Lemma \ref{main-lem}, there are positive integers $d$ and $e$ with $d\leqslant e\leqslant n^2$ such that
\begin{enumerate}
\item [(a)] $a_i(R/J(\H)^{s_0}) = ds_0-e$; and 
\item [(b)] $a_i(R/J(\H)^t) \geqslant dt-e$ for all $t\geqslant r\left\lceil\frac{n}{2}\right\rceil+1$.
\end{enumerate}
We will prove that $a_i(R/J(\H)^s) = ds-e$ for all $s\geqslant n^2$. Indeed, for any $s\geqslant n^2$, by Lemma $\ref{main-lem}$ again, there are positive integers $a$ and $b$ with $a\leqslant b\leqslant n^2$ such that
\begin{enumerate}
\item [(c)] $a_i(R/J(\H)^{s}) = as-b$; and 
\item [(d)] $a_i(R/J(\H)^t) \geqslant at-b$ for all $t\geqslant r\left\lceil\frac{n}{2}\right\rceil+1$.
\end{enumerate}

From $(b)$ and $(c)$ we have $as-b \geqslant ds-e$, or equivalently, 
\begin{equation}\label{MF1}
(a-d)s \geqslant b-e.
\end{equation}
Note that $1\leqslant b,e\leqslant n^2$, so $|b-e|= \max\{b-e,e-b\}\leqslant n^2-1 < s$. Together with the inequality $(\ref{MF1})$, this fact forces $a\geqslant d$. 

Similarly,  by using $(a)$ and $(d)$ we obtain $d\geqslant a$ and
\begin{equation}\label{MF2}
(d-a)s_0 \geqslant e-b.
\end{equation}

Thus, $a=d$. Together this fact with Inequalities $(\ref{MF1})$ and $(\ref{MF2})$, respectively, we get $e\geqslant b$ and $b\geqslant e$, respectively. Hence, $b=e$. Then, $$a_i(R/J(\H)^s) = as-b = ds-e,$$ 
and the proof is complete.
\end{proof}

The following theorem is the main result of this paper.

\begin{thm} \label{T2} Let $\H$ be a unimodular hypergraph with $n$ vertices and rank $r$. Then there is a non-negative integer $e\leqslant \dim R/J(\H)-d(J(\H))+1$ such that $\reg J(\H)^s = d(J(\H))s+e$ for all $s\geqslant r\left\lceil\frac{n}{2}\right\rceil+1$.
\end{thm}
\begin{proof} If $n = 1$ or $r=1$, then $J(\H)$ is a principal ideal, and the theorem holds for these cases. Hence we may assume that $n \geqslant 2$ and $r\geqslant 2$.

Let $\delta:=\dim R/J(\H)$. Then, $\delta  < n$ as $J(\H)\ne 0$.  Since
$$\reg R/J(\H)^t) = \max\{a_i(R/J(\H)^t)+i\mid i=0,\ldots,\delta\} \text{ for } t\geqslant 1,$$
by Theorem $\ref{T1}$ we imply that there are integers $t_0$, $d$ and $e$ with $e \leqslant \delta-d$ such that
\begin{equation}\label{basic-eq}
\reg R/J(\H)^t = dt+e \text{ for all } t\geqslant t_0.
\end{equation}

Note that $\reg J(\H)^t=\reg R/J(\H)^t+1$, so when comparing with \cite[Theorem 5]{K} we deduce that $d = d(J(\H))$ and $e\geqslant -1$. Consequently, $e\leqslant \delta - d(J(\H))$.

Let $k$ be an integer such that $k \geqslant  \max\{t_0,2n^2\}$. Assume that $\reg R/J(\H)^k = a_i(R/J(\H)^k)+i$, for some $0\leqslant i\leqslant \delta$. By Lemma $\ref{main-lem}$, there are non-negative integers $a$ and $b'$ with $a\leqslant b'\leqslant n^2$ such that
$a_i(R/J(\H)^k) = ak-b'$ and $a_i(R/J(\H)^t) \geqslant at-b'$ for all $t\geqslant r\left\lceil\frac{n}{2}\right\rceil+1$.

Let $b:=-b'+i$. Then, $-n^2\leqslant b\leqslant i-a \leqslant \delta-a$,  $\reg R/J(\H)^k = ak+b$ and
\begin{equation}\label{F0}
\reg R/J(\H)^t \geqslant at+b \text{ for all } t\geqslant r\left\lceil\frac{n}{2}\right\rceil+1.
\end{equation}

As $\reg R/J(\H)^k = ak+b = dk+e$, we have $(d-a)k = b-e$. Together with the fact $|b-e|\leqslant |b|+e\leqslant n^2+\delta+1\leqslant n^2+n <2n^2\leqslant k$, it forces $d = a$, and so $b=e$. Thus, $(\ref{F0})$ becomes
\begin{equation}\label{F1}
\reg R/J(\H)^t \geqslant dt+e \text{ for all } t\geqslant r\left\lceil\frac{n}{2}\right\rceil+1.
\end{equation}

We next prove that these inequalities are in fact equalities and therefore the theorem follows because $\reg J(\H)^t=\reg R/J(\H)^t+1$ for $t\geqslant 1$.

In order to prove this, let $s$ be an integer such that $s\geqslant r\left\lceil\frac{n}{2}\right\rceil+1$. By the argument above, there are integers $c$ and $f$ with $f\leqslant \delta-c$ such that
$\reg R/J(H)^s =cs+f$ and
\begin{equation}\label{F2}
\reg R/J(H)^t \geqslant ct+f \text{ for all } t\geqslant r\left\lceil\frac{n}{2}\right\rceil+1.
\end{equation}

From $(\ref{basic-eq})$ and $(\ref{F2})$ we get $c\leqslant d$. Since $\reg R/J(H)^s =cs+f$, by $(\ref{F1})$ we have $cs+f\geqslant ds+e$, so $(d-c)s\leqslant f-e$. As $c\leqslant d$, it follows that $f \geqslant e$. In particular, $f\geqslant -1$. Observe that $n \leqslant s$ as $r\geqslant 2$, so that
$$f-e\leqslant (\delta-c)+1 \leqslant \delta < n\leqslant s.$$ 
Together with $(d-c)s\leqslant f-e$ and $c\leqslant d$, this fact forces $d-c=0$, i.e. $d = c$. Combining this equality with $(\ref{basic-eq})$ and $(\ref{F2})$ we get $e\geqslant f$, and therefore $e = f$.

Finally, because $d=c$ and $e=f$, we have $\reg R/J(\H)^s = cs+f = ds+e$, and the proof of the theorem is complete.
\end{proof}

\begin{cor} Let $G$ be a bipartite graph with $n$ vertices. Then, $\reg J(G)^s$ is a linear function in $s$ for all $s\geqslant n+2$.
\end{cor}

\subsection*{Acknowledgment}  This work is partially supported by NAFOSTED (Vietnam) under the grant number 101.04-2015.02. The first author is also partially supported by Thai Nguyen university of Science under the grant number \DJ H2016-TN06-03.


\begin{thebibliography}{99}

\bibitem {ABS}  A. Alilooee, S. Beyarslan and S. Selvaraja, {\it Regularity of Powers of Unicyclic Graphs},  arXiv:1702.00916.

\bibitem {Ba}A. Banerjee, {\it The regularity of powers of edge ideals}, J. Algebr. Comb. {\bf 41}(2015), no. 2, 303 -- 321.

\bibitem {Berge}  C. Berge, Hypergraphs: combinatorics of finite sets. North-Holland, New York, 1989.

\bibitem {Ber}  D. Berlekamp, {\it Regularity defect stabilization of powers of an ideal}. Math. Res. Lett. {\bf 19} (2012), no. 1, 109 -- 119.

\bibitem {BHT}   S. Beyarslan, H{$\rm \grave{a}$} H.T.  and  T. N. Trung. {\it Regularity of powers of forests and cycles}, J.
Algebraic Combin., {\bf 42} (2015), no. 4, 1077 -- 1095.


\bibitem {C} M. Chardin, {\it Powers of ideals and the cohomology of stalks and fibers of morphisms}, Algebra Number Theory {\bf 7} (2013), no. 1, 1 -- 18.

\bibitem {C1} M. Chardin, {\it Regularity stabilization for the powers of graded $\mathfrak M$-primary ideals}. Proc. Amer. Math. Soc. {\bf 143} (2015), no. 8, 3343--3349.

\bibitem {Ckk} S. Cutkosky, {\it Irrational asymptotic behaviour of Castelnuovo-Mumford regularity}, J. Reine Angew. Math. {\bf 522} (2000), 93 -- 103.


\bibitem {CHT} D. Cutkosky, J. Herzog and N. V. Trung, {\it Asymptotic behavior of the Castelnuovo-Mumford regularity}, Compositio Math. {\bf 118} (1999), 243 -- 261.

\bibitem {EH} D. Eisenbud and J. Harris, {\it Powers of ideals and fibers of morphisms}, Math. Res. Lett. {\bf 17} (2010), no. 2, 267 -- 273.

\bibitem {EU} D. Eisenbud and  B. Ulrich, {\it Notes on regularity stabilization}, Proc. Amer. Math. Soc. {\bf 140} (2012), no. 4, 1221 -- 1232.

\bibitem {Ha} H. T. H\`a, {\it Asymptotic linearity of regularity and $a^*$-invariant of powers of ideals}, Math. Res. Lett. {\bf 18} (2011), no. 1, 1--9.

\bibitem  {HTT} H. T. H{$\rm \grave{a}$}, N. V. Trung and T. N. Trung, {\it Depth and regularity of powers of sums of ideals},  Math. Z. {\bf 282} (2016), no. 3-4, 819--838.

\bibitem  {HT1} N. T. Hang and T. N. Trung, {\it The behavior of depth functions of powers of cover ideals of unimodular hypergraphs},  Ark. Math. (preprint 2016 to appear).
 

\bibitem  {HHT} J. Herzog, T. Hibi and N. V. Trung, {\it Vertex cover algebras of unimodular hypergraphs}, Proc. Amer.
Math. Soc. {\bf 137} (2009), 409--414.

\bibitem  {HKTT} L. T. Hoa, K. Kimura, N. Terai and T. N. Trung, {\it Stability of depths of symbolic powers of Stanley-Reisner ideals}, J. Algebra {\bf 473} (2017), 307--323. 
 
\bibitem  {HT2} L. T. Hoa and T. N. Trung, {\it Partial Castelnuovo-Mumford regularities of sums and intersections of powers of monomial ideals}, Math. Proc. Cambridge Philos Soc. {\bf 149} (2010), 1--18.

\bibitem  {HO}  M. Hochster, Cohen-Macaulay rings, combinatorics, and simplicial complexes, in B. R. Mc- Donald and R. A. Morris (eds.), Ring theory II, Lect. Notes in Pure and Appl. Math. {\bf 26}, M. Dekker, 1977, 171--223.

\bibitem  {JNS} A. V. Jayanthan, N. Narayanan and  S. Selvaraja, {\it Regularity of Powers of Bipartite Graphs}, 	arXiv:1609.01402.


\bibitem {K} V. Kodiyalam, {\it Asymptotic behaviour of Castelnuovo-Mumford regularity}, Proc. Amer. Math. Soc. {\bf 128} (2000), 407 -- 411.


\bibitem  {MT1} N. C. Minh and N. V. Trung, {\it Cohen-Macaulayness of powers of two-dimensional squarefree monomial ideals}, J. Algebra {\bf 322} (2009), 4219--4227.

\bibitem {MS} E. Miller and B. Sturmfels, Combinatorial commutative algebra. Springer, 2005.

\bibitem  {MT2} N. C. Minh and T. N. Trung, {\it Regularity of symbolic powers and Arboricity of matroids},  arXiv:1702.04491.

\bibitem {S} A. Schrijver, Theory of linear and integer programming, John Wiley $\&$ Sons, 1998.

\bibitem {T} Y. Takayama, {\it Combinatorial characterizations of generalized Cohen-Macaulay monomial ideals}, Bull. Math. Soc. Sci. Math. Roumanie (N.S.) {\bf 48} (2005), 327--344.

\bibitem{TW} N. V. Trung and H. Wang, {\it On the asymptotic linearity of Castelnuovo-Mumford regularity}, J. Pure Appl. Algebra {\bf 201}(2005), no. 1--3, 42--48 .

\end{thebibliography}
\end{document}